\documentclass[12pt]{amsart}

\usepackage[english]{babel}
\usepackage[utf8]{inputenc}
\usepackage[T1]{fontenc}
\usepackage{amsmath}
\usepackage{amsthm}
\usepackage{amssymb}
\usepackage{graphicx}
\graphicspath{{./images/}}
\usepackage{enumitem}
\usepackage{float}

\newtheorem{definition}{Definition}
\newtheorem{prop}{Proposition}

\newtheorem{remark}{Remark}
\newtheorem{exe}{Example}
\newtheorem{lemma}{Lemma}
\newtheorem{theorem}{Theorem}
\newtheorem{cor}{Corollary}

\DeclareMathOperator{\arcsinh}{arcsinh}
\DeclareMathOperator{\diam}{diam}
\DeclareMathOperator{\inj}{inj}
\DeclareMathOperator{\length}{l}

\title[Estimates for low Steklov eigenvalues]{Estimates for low Steklov eigenvalues of surfaces with several boundary components
}
\author{Hélène Perrin}
\date{}
\thanks{The author was supported by the Swiss National Science Fundation (SNSF), grant 200021\_196894.}

\begin{document}

\begin{abstract}
In this article, we give computable lower bounds for the first non-zero Steklov eigenvalue $\sigma_1$ of a compact connected 2-dimensional Riemannian manifold $M$ with several cylindrical boundary components. These estimates show how the geometry of $M$ away from the boundary affects this eigenvalue. They involve geometric quantities specific to manifolds with boundary such as the extrinsic diameter of the boundary. In a second part, we give lower and upper estimates for the low Steklov eigenvalues of a hyperbolic surface with a geodesic boundary in terms of the length of some families of geodesics. This result is similar to a well known result of Schoen, Wolpert and Yau for Laplace eigenvalues on a closed hyperbolic surface.
\end{abstract}

\maketitle

\begin{section}{Introduction}

We study lower bounds for low Steklov eigenvalues of a compact connected 2-dimensional Riemannian manifold with several boundary components. Few lower bounds are known for the first non-zero Steklov eigenvalue $\sigma_1$. For a Riemannian manifold with connected boundary, there are generalizations (see e.g. \cite{Escobar1}, \cite{Escobar2} and \cite{Xiong2022})  of a result of Payne \cite{Payne1970} of 1970 saying that $\sigma_1$ of a convex domain in the plane is bounded from below by the minimum curvature of its boundary. In a general setting, Escobar \cite{Escobar1} has given a lower bound depending on an isoperimetric constant and the first non-zero eigenvalue of a Robin problem (see also \cite{AsmaMiclo} for lower bounds depending on eigenvalues of auxiliary problems). In \cite{JammesStek}, Jammes gives lower bounds in terms of isoperimetric constants. This result has been generalized by Hassannezhad and Miclo \cite{AsmaMiclo} for higher eigenvalues. These lower bounds however are not easily computable. In \cite{CGH}, Colbois, Girouard and Hassannezhad show that under some assumptions on the geometry of the boundary and near the boundary, Steklov eigenvalues are well approximated by the Laplace eigenvalues of the boundary. But when a connected Riemannian manifold has $b\geq2$ boundary components, such estimates do not give lower bounds for the $b$ first eigenvalues of the Steklov problem.
% since $0$ is an eigenvalue of the Laplacian on the boundary with multiplicity $b$.

For obtaining lower bounds, conditions on the intrinsic geometry of the boundary as well as conditions on the geometry near the boundary are expected. But even if the boundary and the geometry of $M$ near the boundary are fixed, $\sigma_1$ is not bounded below if the boundary has multiple connected components, as shows the case of a right cylinder whose first eigenvalue tends to zero as its height goes to infinity.

In this article, we give explicit estimates for the $b$ first Steklov eigenvalues of some families of compact connected 2-dimensional Riemannian manifolds with $b\geq2$ boundary components having each one a neighborhood which is a right or a hyperbolic cylinder. This strong assumption on the geometry near the boundary allows us to focus on how the geometry of the manifold away from the boundary affects these eigenvalues.
%We will see that constraints on the geometry of some region of the manifold are enough for bounding the first non-zero Steklov eigenvalue $\sigma_1$ from below.
The first result is an explicit lower bound for $\sigma_1$ of a 2-dimensional Riemannian manifold with a cylindrical boundary. It does not require any assumption on the Gaussian curvature and involves the following quantity.

\begin{definition}
Let $(M,g)$ be a compact connected $2$-dimensional Riemannian manifold with $b\geq 2$ boundary components. We consider the family of curves (not necessarily connected) not intersecting $\partial M$ and dividing $M$ into two connected components, each containing at least one connected component of $\partial M$. We let  $C(M)$ denote this family of curves and define
\[
\length(M):=\inf\{l(c): c\in C(M)\}
\]
where $l(c)$ is the length of the curve $c$.
\label{def1}
\end{definition}
We can now state the result.

\begin{theorem}
Let $(M,g)$ be a compact connected $2$-dimensional Riemannian manifold with a boundary having $b\geq 2$ components of length $a$. Assume that the boundary $\partial M=\Sigma_1\cup\dots\cup\Sigma_b$ has a neighborhood $V(\partial M)$ which is isometric to the union of disjoint right cylinders $\cup_{i=1}^{b}\Sigma_i\times [0,1)$. We have
\begin{align*}
\sigma_1(M)\geq \frac{\min\{\length(M),1\}^2}{2(b-1)a|M|}.
\label{est1}
\end{align*}
\label{prop1}
\end{theorem}
Examples \ref{excyl} and \ref{ex3} in Section \ref{sec3} show that the exponent of the geometric quantities involved in the lower bound are optimal. Another natural question to ask for evaluating a lower bound is how close to $\sigma_1$ it is. We construct a family of surfaces which shows that the presence of the area of the manifold in the denominator makes the lower bound given in Theorem \ref{prop1} sometimes inaccurate since it can go to zero while $\sigma_1$ is constant.

For surfaces whose Gaussian curvature is bounded below, we succeeded in removing the depency of the area from the lower bound. This estimate involves the extrinsic diameter of the boundary and the injectivity radius of a certain subset of $M$.
\begin{definition}
\label{defIntro}
Let $(M,g)$ be a compact connected Riemannian manifold with boundary $\partial M$.
\begin{enumerate}
\item The extrinsic diameter of the boundary is
\[
\diam_M(\partial M)=\max\{d(x,y)| x, y\in \partial M\},
\]
where $d(x,y)$ denotes the distance on $M$ induced by $g$.
\end{enumerate}
For simplification, we will omit the term "extrinsic" and call it the diameter of the boundary. Assume now that the boundary $\partial M=\Sigma_1\cup\dots\cup\Sigma_b$ has a neighborhood $V(\partial M)$ which is isometric to the union of disjoint right cylinders $\cup_{i=1}^{b}\Sigma_i\times [0,1)$.

\begin{enumerate}[resume]
\item Let $\Gamma$ be the subset of $M$
\begin{multline*}
\Gamma=\{x\in M, \exists p,q \in \partial M \text{ and a length minimising geodesic }\gamma \\\text{ between p and q such that } x\in \gamma\}.
\end{multline*}
We denote $\inj_{\partial M}(M)$ the injectivity radius of $\Gamma \setminus V(\partial M)\subset M$, that is
\[
\inj_{\partial M}(M)=\inj(\Gamma \setminus V(\partial M))=\min\{\inj_M(x): x\in \Gamma\setminus V(\partial M)\}.
\]
We note that $\inj_{\partial M}(M)\leq 1$.
\end{enumerate}
\end{definition}

\begin{theorem}
Let $(M,g)$ be a compact connected $2$-dimensional Riemannian manifold with a boundary having $b\geq 2$ boundary components of length $a$.
Assume that the boundary $\partial M=\Sigma_1\cup\dots\cup\Sigma_b$ has a neighborhood $V(\partial M)$ which is isometric to the union of disjoint right cylinders $\cup_{i=1}^{b}\Sigma_i\times [0,1)$.
Assume that the Gaussian curvature of $M$ satisfies $K(p)\geq \kappa$ for all $p\in M$, where $\kappa$ is a negative constant, and assume that $a\leq\diam_M(\partial M)$. Then we have an explicit positive constant $C(\kappa,b)$ such that
\[
\sigma_1(M)\geq C(\kappa,b)\frac{\inj_{\partial M}(M)}{a\diam_{M}(\partial M)}.
\]
\label{borneinf2}
\end{theorem}

As for Theorem \ref{prop1}, we show that the exponent of the geometric quantities involved in Theorem \ref{borneinf2} cannot be improved (see Remark \ref{remopt}).
With the stronger assumption that the injectivity radius is bounded from below at each point of $M$ outside the cylindrical neighborhood of the boundary, Theorem \ref{borneinf2} can also be obtained from the combination of the lower bound given in \cite{PStek2018} for $\sigma_1$ of the Steklov problem on graphs and the discretization process described in \cite{colbois2016steklov}.

We note that results for surfaces with cylindrical boundary are significant since they can be used for deducing results for any manifolds with boundary by using quasi-isometries as it has been done in \cite{colbois2017compact} (see Theorem 1.1). Since surfaces that are conformal inside and isometric on the boundary are Steklov isospectral, the results also give lower bounds for surfaces that are conformal inside and isometric on the boundary to one with cylindrical boundary.

In a second part, we give an upper and lower estimate for the $b$ first Steklov eigenvalues of compact hyperbolic surfaces with $b$ geodesic boundary components. It shows that these eigenvalues are equivalent to the length of some separating curves of the manifold. The result is similar to a result of Schoen, Wolpert and Yau \cite{SWY} for Laplace eigenvalues. However, the family of curves that are relevant is different.

\begin{definition}
Let $M$ be a compact hyperbolic surface with $b\geq2$ geodesic boundary components. For $1\leq n\leq b-1$, we consider the family of curves which consist of a union of disjoint simple closed geodesics, not intersecting $\partial M$, and dividing $M$ into $n+1$ connected components, each containing at least one connected component of $\partial M$. We denote $C_n(M)$ the family of such curves. If $C_n(M)\not=\emptyset$, we define
\[
\length_n(M):=\min\{l(c): c\in C_n(M)\}
\]
where $l(c)$ is the length of the curve $c$.
\end{definition}
We have the following result.
\begin{theorem}
Let $M$ be a hyperbolic surface of genus $g$ with $b\geq 2$ geodesic boundary components, each of them having length $a\leq2\arcsinh(1)$. Assume that $g\not =0$ or $b>3$. There exists a constant $C_1$ depending only on $g$ and $b$ and a universal constant $C_2$ such that for $1\leq n<\lceil \frac{b}{2}\rceil$ we have
\[
C_1\length_n^2\leq\sigma_n\leq C_2 \frac{\length_n}{a}.
\]
The inequality is also true for $\lceil \frac{b}{2}\rceil\leq n <b$ if $C_n(M)\not=\emptyset$ and  there exists $c\in C_n(M)$ such that each simple closed geodesic of $c$ is of length $l\leq L_{g+b}$, where $L_{g+b}=4(3(g+b)-3)\log(\frac{8\pi(g+b-1)}{3(g+b)-3})$.
\label{thmhyp}
\end{theorem}

If $a$ becomes small, we see that the upper bound becomes big, but we are also able to show that for $0\leq n<b$, $\sigma_n$  is bounded above by $\frac{1}{\arctan(\frac{1}{\sinh{\frac{a}{2}}})}\leq\frac{2}{\pi}$.

\begin{remark}
It is possible to obtain results similar to Theorems  \ref{prop1}, \ref{borneinf2} and  \ref{thmhyp} without the assumption that all the boundary components have the same length. In this case, we have to replace $a$ by the maximum length of the boundary components in Theorems \ref{prop1} and \ref{borneinf2}. In Theorem \ref{thmhyp}, we have to replace $a$ in the upper bound by the minimum length of the boundary components and make the assumption that the maximum length of the boundary components is $\leq 2 \arcsinh(1)$. The results are obtained by slightly modifying the proofs given in Section \ref{sec3}.
\end{remark}

An important tool for obtaining our results is estimating isoperimetric constants in an improved statement of a lower bound given by Jammes for the first non-zero Steklov eigenvalue. The strategy of estimating isoperimetric constants has been used in the past for obtaining lower bounds for the first non-zero Laplace eigenvalue on closed surfaces (see e.g. \cite{Buser77} and \cite{SWY}). We also use comparisons with mixed problems.

The article is structured as follows. In Section \ref{chap2} we introduce mixed problems and Cheeger-type estimates for Steklov eigenvalues. The main results are proved in Section \ref{sec3}, which is divided in two parts. In the first part, we prove a generalization of Theorem \ref{prop1} and then use it to prove Theorem \ref{borneinf2}. In the second part, we recall some useful properties of hyperbolic surfaces and prove Theorem \ref{thmhyp}.

\end{section}

\begin{section}{Cheeger-type estimates and mixed problems}
\label{chap2}
\begin{subsection}{Steklov eigenvalues}
Let $(M,g)$ be a compact connected Riemannian manifold with Lipschitz boundary $\partial M$. The Steklov problem on $M$ is the eigenvalue problem
\begin{equation*}
\begin{cases}
\triangle u=0\\
\frac{\partial u}{\partial \nu}=\sigma u
\end{cases}
\label{Stekprob}
\end{equation*}
where $\sigma$ is the spectral parameter.
The Stekov eigenvalues form a sequence $0=\sigma_0<\sigma_1\leq\sigma_2\leq\dots\nearrow$. 
They can be characterized variationally as follows:
\begin{equation*}
\sigma_k(M)=\min_{E\in V_k}\max_{0\not=u\in E} R(u),
\end{equation*}
where $V_k$ is the set of all $k+1$ dimensional subspaces of the Sobolev space $H^1(M)$, and $R(u)$ is the Rayleigh quotient associated to the Steklov problem,
\begin{equation*}
R(u)=\frac{\int_{M} |\nabla u|^2dv_g}{\int_{\partial M} u^2 dS_g}.
\end{equation*}

There is a connection between Steklov eigenvalues of a Riemannian manifold $(M,g)$ with boundary $\partial M$ and eigenvalues of mixed problems on a Lipschitz open subset $A\subset M$ containing $\partial M$. 
Given a Lipschitz open subset $A\subset M$ such that $\partial M\subset A$, we denote $\partial A$ the topological boundary of $A$ as a subset of $M$. The mixed Steklov-Neumann problem on $A$ is
\[
\begin{cases}
\triangle u=0 &\text{ in }A,\\
\frac{\partial u}{\partial \nu}=\sigma u &\text{ on }A\cap \partial M,\\
\frac{\partial u}{\partial \nu}=0 &\text{ on }\partial A,
\end{cases}
\]
and the mixed Steklov-Dirichlet problem on $A$ is
\[
\begin{cases}
\triangle u=0 &\text{ in }A,\\
\frac{\partial u}{\partial \nu}=\sigma u &\text{ on }A\cap \partial M,\\
u=0 &\text{ on }\partial A.\\
\end{cases}
\]
The eigenvalues of the mixed Steklov-Neumann problem form a discrete sequence $0=\sigma_0^N(A)\leq\sigma_1^N(A)\leq \sigma_2^N(A)\leq \dots\nearrow$ and the eigenvalues of the mixed Steklov-Dirichlet problem form a discrete sequence $0<\sigma_0^D(A)\leq\sigma_1^D(A)\leq\sigma_2^D(A)\leq\dots\nearrow$.

The eigenvalues satisfy
\begin{equation}
\sigma_k^N(A)\leq \sigma_k(M)\leq\sigma_k^D(A).
\label{compvp}
\end{equation}
The proof of this inequality follows from a comparison between the Rayleigh quotients of these problems, see \cite{CGG} for more details.
\end{subsection}

\begin{subsection}{Cheeger-type estimates}
In 1969, J. Cheeger \cite{Cheeger1970} gave a lower bound in term of an isoperimetric constant for the first non-zero Laplace eigenvalue of a compact Riemannian manifold.  A similar estimate for the first non-zero Steklov eigenvalue was shown by P. Jammes in 2015 \cite{JammesStek}. We give an improvement of this result that we use to obtain the explicit lower bounds presented in this article.

\begin{definition}
\begin{enumerate}
We define the following geometric constants:
\item
\[
h_1(M):=\inf_{|D|\leq\frac{|M|}{2}}\frac{|\partial D|}{|D|},
\]
\item
\[
h_2(M):=\inf_{|D|\leq\frac{|M|}{2}}\frac{|\partial D|}{|D\cap \partial M|},
\]
where in both cases, $D$ is taken among the domains of $M$ satisfying $D\cap \partial M \not =\emptyset$, and such that $M\setminus D$ is also connected and intersects $\partial M$.
\end{enumerate}
\label{defcst1}
\end{definition}

\begin{remark} The set $\partial D$ is the topological boundary of the open subset $D$ of the manifold with boundary $M$; this set does not contain $D\cap\partial M$.
\end{remark}
\begin{remark}
Jammes defines  two constants in a similar way but the domain $D$ is only required to satisfy $|D|\leq\frac{|M|}{2}$.
\end{remark}

\begin{prop} Let $(M,g)$ be a compact Riemannian manifold with boundary $\partial M$. We have
\begin{equation*}
\sigma_1(M)\geq\frac{ h_1(M)\cdot h_2(M)}{4}.
\label{Cheeger1}
\end{equation*}
\label{propCheeg1}
\end{prop}

This is the result of Jammes but with slightly modified constants. It is obtained by modifying the conclusion of Jammes's proof. Example \ref{exSurfrev} below shows that in dimension $2$ this inequality is stronger than the one given by Jammes where $D$ is only required to satisfy $|D|\leq\frac{|M|}{2}$ in the isoperimetric constants.  Another situation where the constants $h_1$ and $h_2$ will not go to zero while the constants of Jammes do is Example 4.5 of \cite{CGGS}.

\begin{exe} Let $C$ be a 2-dimensional right cylinder in $\mathbb{R}^3$ whose base contains a line segment. We consider the surfaces obtained by gluing a surface of revolution containing a thin collapsing cylinder on the middle of the flat part of $C$, as shown in Figure \ref{ex1}.
These surfaces are all Steklov isospectral to $C$ (see \cite{CGG}, Appendix A, and \cite{brissonMT} for more details). However, Jammes's constants tend to zero as the thin passage collapses. In contrast, the constants $h_1$ and $h_2$ that we use remain bounded (see Lemma \ref{lemma1}).
\label{exSurfrev}

\begin{figure}[H]
\includegraphics[scale=0.25]{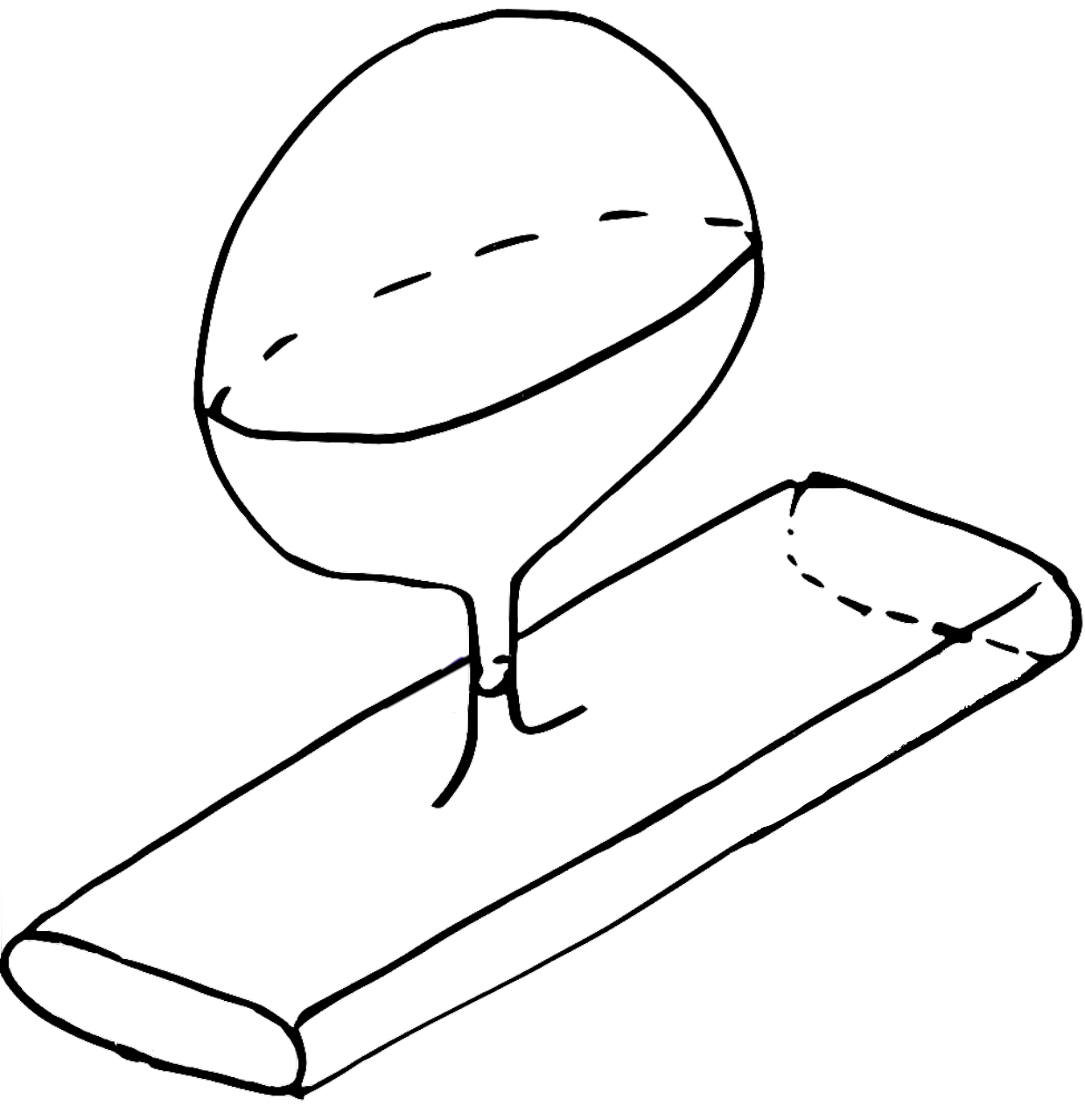}
\centering
\caption{A cylinder on which we have glued a surface of revolution}
\label{ex1}
\end{figure}
\end{exe}

\begin{proof}[Proof of Proposition \ref{propCheeg1}]
Let $u$ be an eigenfunction associated to the first non-zero Steklov eigenvalue on $M$. We define
\[
D(t):=\{x\in M, u(x)>t\}.
\]
Without loss of generality, we can assume $|D(t)|\leq\frac{|M|}{2}$ for all $t\geq 0$. From the proof of Jammes's result, which is similar to the classical proof of Cheeger, we have
\[
\sigma_1(M)\geq\frac{1}{4}\min_{t\geq 0}\frac{|\partial D(t)|}{|D(t)|}\cdot\min_{t\geq0}\frac{|\partial D(t)|}{|D(t)\cap \partial M|}.
\]
Since $u$ is harmonic and not constant, it follows from the maximum principle that  each connected component of $M\setminus \{u^{-1}(t)\}$ intersects $\partial M$. Therefore, the inequalities $\min_{t\geq 0}\frac{|\partial D(t)|}{|D(t)|}\geq\inf_{|D|\leq\frac{|M|}{2}}\frac{|\partial D|}{|D|}$ and $\min_{t\geq 0}\frac{|\partial D(t)|}{|D(t)\cap \partial M|}\geq \inf_{|D|\leq\frac{|M|}{2}}\frac{|\partial D|}{|D\cap \partial M|}$ are true if the infima are taken among all sets $D$ such that each connected component of $D$ and of $M\setminus D$ intersects $\partial M$. Finally, as observed by S.-T. Yau in \cite{YauIsopcon}, we can assume that both $D$ and $M\setminus D$ are connected.
\end{proof}

\begin{remark} It has been showed (see e.g. \cite{BuserHawaii}, Theorem 1.14) that the lower bound of Cheeger for the first non-zero Laplace eigenvalue is sharp. It would be interesting to know if the lower bound of Jammes is sharp too.
\end{remark}

\begin{remark}
Given a domain $A$ in $M$ such that $A\cap \partial M\not =\emptyset$, we define the constants $h_1(A)$ and $h_2(A)$ in the same way as for $M$, by replacing $M$ by $A$ and $\partial M$ by $\partial M\cap A$ in the conditions that $D$ has to satisfy. The same proof as for Proposition \ref{propCheeg1} shows that $\sigma_1^N(A)\geq\frac{ h_1(A)\cdot h_2(A)}{4}$.
\label{remMixed}
\end{remark}

In the construction described in Example \ref{exSurfrev}, if we glue two surfaces of revolution of equal volume on $C$ instead of one and let grow the volume of these surfaces of revolution, we see that $h_1$ tends to zero by choosing a domain that contains one of the two surfaces of revolution. This shows that in this case, the estimate of Proposition $2$ is not equivalent to the first non-zero Steklov eigenvalue, which is constant since the surfaces obtained are isospectral to $C$.

The following proposition shows a way of improving Proposition \ref{propCheeg1}.

\begin{prop}
Let $(M,g)$ be a compact Riemannian manifold with boundary $\partial M$. For any domain $A$  in $M$ such that $\partial M\subset A$, we have
\[
\sigma_1(M)\geq\frac{ h_1(A)\cdot h_2(A)}{4}.
\]
\label{propCheeg2}
\end{prop}
\begin{proof}
The proof follows from the comparison (\ref{compvp}) between Steklov and mixed Steklov-Neumann eigenvalues and Remark \ref{remMixed}.
\end{proof}
This estimate is interesting if we can find domains such that $h_1$ and $h_2$ are bounded below. Finally, having in mind Question 4.6 of \cite{CGGS}, we remark that by taking the supremum over the domains $A$, a new constant is defined. It is more acurate than the product $h_1(M)\cdot h_2(M)$ but difficult to calculate.
\end{subsection}
\end{section}

\begin{section}{Explicit estimates for Steklov eigenvalues}
\label{sec3}
\begin{subsection}{Lower bounds for $\sigma_1$ of surfaces with several cylindrical boundary components}

We recall the following estimate for Steklov eigenvalues of surfaces with cylindrical boundary components, which follows directly from the comparison (\ref{compvp}) with eigenvalues of mixed Steklov-Neumann and Steklov-Dirichlet problems on the union of the cylindrical boundary nieghborhoods, and the explicit calculation of these.
\begin{lemma} Let $(M,g)$ be a compact $2$-dimensional Riemannian manifold with $b\geq 1$ boundary components having length $a$. Assume that the boundary $\partial M=\Sigma_1\cup\dots\cup\Sigma_b$ has a neighborhood $V(\partial M)$ which is isometric to the union of disjoint right cylinders $\cup_{i=1}^{b}\Sigma_i\times [0,L)$. The Steklov eigenvalues $\sigma_k$ of $M$ satisfy
\begin{align*}
0\leq\sigma_k\leq \frac{1}{L}
\end{align*}
if $k<b$, and
\begin{align*}
\frac{2\pi j}{a}\tanh(\frac{2\pi j}{a}L)\leq\sigma_k\leq \frac{2\pi j}{a}\coth(\frac{2\pi j}{a}L)
\end{align*}
if $(2j-1)b\leq k< (2j+1)b$, where $j\in\mathbb{N}^*$.
\label{compcyl}
\end{lemma}
We see that if $b=1$, $\sigma_1$ is bounded below by $\frac{2\pi}{a}\tanh(\frac{2\pi}{a}L)$, but if $b>1$ this lemma does not give a lower bound for $\sigma_1$. Therefore, our results concern only the case $b\geq2$ which is interesting.

Theorem \ref{prop1} is in fact a particular case of a more general result (Theorem \ref{thm0} below) that involves domains of $M$ containing the cylindrical neighborhood of $\partial M=\Sigma_1\cup\dots\cup\Sigma_b$, which is in this result assumed to be isometric to the union of disjoint right cylinders $\cup_{i=1}^{b}\Sigma_i\times [0,L)$. We note that given such a domain $A$, we can define $\length(A)$ in the same way as we have defined $\length(M)$ in Definition \ref{def1} by considering curves that divide $A$ into two connected components without intersecting $\partial M$. The reason for proving this result instead of Theorem 1 is that it is needed in the proof of Theorem \ref{borneinf2}.

\begin{theorem}
Let $(M,g)$ be a compact connected $2$-dimensional Riemannian manifold with a boundary having $b\geq 2$ components of length $a$. Assume that the boundary $\partial M=\Sigma_1\cup\dots\cup\Sigma_b$ has a neighborhood $V(\partial M)$ which is isometric to the union of disjoint right cylinders $\cup_{i=1}^{b}\Sigma_i\times [0,L)$. For any domain $A$ in $M$ such that $V(\partial M)\subset A$ (possibly $A=M$), we have
\begin{align*}
\sigma_1(M)\geq \frac{\min\{\length(A),L\}^2}{2(b-1)a|A|}.
\end{align*}
\label{thm0}
\end{theorem}
The proof of Theorem \ref{thm0} involves estimating the constants $h_1$ and $h_2$ of compact connected $2$-dimensional manifolds with cylindrical boundary.

\begin{lemma}
Let $(M,g)$ be a compact connected $2$-dimensional Riemannian manifold with $b\geq 2$ boundary components having length $a$. Assume that the boundary $\partial M=\Sigma_1\cup\dots\cup\Sigma_b$ has a neighborhood $V(\partial M)$ which is isometric to the union of disjoint right cylinders $\cup_{i=1}^{b}\Sigma_i\times [0,L)$.
Let $A$ be a domain in $M$ such that $V(\partial M)\subset A$ (me may have $A=M$). We have the following estimates of $h_1$ and $h_2$:
$$h_1(A)\geq \frac{2\min\{\length(A),L\}}{|A|},$$
$$h_2(A)\geq \frac{\min\{\length(A),L\}}{(b-1)a}.$$
\label{lemma1}
\end{lemma}
\begin{proof}
We recall that
\[
h_1(A)=\inf \frac{|\partial D|}{|D|}
\]
where the infimum is taken among all domains satisfying $|D|\leq\frac{|A|}{2}$, $D\cap\partial M\not =\emptyset$ and such that $A\setminus D$ is also connected and intersects $\partial M$. Given such a domain $D$ the following situations can happen.
\begin{enumerate}

\item $\partial D$ intersects a boundary component $\Sigma_i$ and is contained in the cylindrical neighborhood of $\Sigma_i$. If $|\partial D|\geq L$, the fact that $aL<|A|$ gives $\frac{|\partial D|}{|D|}\geq\frac{|\partial D|}{aL}\geq\frac{L}{aL}=\frac{1}{a}>\frac{L}{|A|}$. If $|\partial D|<L$, we know from the isoperimetric inequality that the domain $D$ minimising $\frac{|\partial D|}{|D|}$ is the half-disk with radius $r=\frac{|\partial D|}{\pi}$ and area $\frac{|\partial D|^2}{2\pi}$. This gives $\frac{|\partial D|}{|D|}\geq |\partial D|\cdot\frac{2\pi}{|\partial D|^2}=\frac{2\pi}{|\partial D|}>\frac{2\pi}{L}>\frac{2\pi a}{|A|}\geq\frac{2\pi l(A)}{|A|}>\frac{2\length(A)}{|A|}$.

\item $\partial D$ intersects a boundary component $\Sigma_i$ but $D$ is not contained in the cylindrical neighbourhood of $\Sigma_i$. The length of the curve $\partial D$ between its extremity in $\Sigma_i$ and the point where it leaves the cylindrical neighborhood is greater or equal to $L$. Hence, we have $\frac{|\partial D|}{|D|}\geq \frac{2L}{|A|}$.
\item $\partial D$ contains a curve of $C(A)$. Since $\length(A)$ is the minimal length of such a curve, $|\partial D|\geq \length(A)$. Moreover, $D$ satisfies $|D|\leq \frac{|A|}{2}$. Hence we have $\frac{|\partial D|}{|D|}\geq \frac{2\length(A)}{|A|}$.
\end{enumerate}
In each case, we have either $\frac{|\partial D|}{|A|}\geq\frac{2\length(A)}{|A|}$ or  $\frac{|\partial D|}{|A|}\geq\frac{2L}{|A|}$. Since we have considered all possible cases, we conclude that  $ h_1\geq\frac{2\min\{\length(A),L\}}{|A|}$.

We now estimate $h_2(A)$. We recall that
\[
h_2(A):=\inf \frac{|\partial D|}{|D\cap \partial M|}
\]
where the infimum is taken among all domains satisfying $|D|\leq\frac{|A|}{2}$, $D\cap\partial M\not =\emptyset$ and such that $A\setminus D$ is also connected and intersects $\partial M$. Given such a domain $D$ the following situations can happen.

\begin{enumerate}
\item $\partial D$ intersects a boundary component $\Sigma_i$ and $D$ is contained in the cylindrical neighborhood of $\Sigma_i$. Since the the complement of $D$ in $M$ is connected, $\partial D$ is homotopic to $D\cap \Sigma_i$. Since $D\cap\Sigma_i$ is a geodesic arc and the cylindrical neighborhood has zero curvature, $|\partial D|\geq |D\cap\Sigma_i|= |D\cap\partial M|$ and finally $\frac{|\partial D|}{|D\cap \partial M|}\geq 1$.
\item $\partial D$ intersects a boundary component $\Sigma_i$ but $D$ is not contained in the cylindrical neighborhood of $\Sigma_i$. The length of the curve $\partial D$ between its extremity in $\Sigma_i$ and the point where it leaves the cylindrical neighborhood is greater or equal to $L$. Hence, we have $\frac{|\partial D|}{|D\cap \partial M|}\geq \frac{2L}{ba}\geq\frac{L}{(b-1)a}$.
\item $\partial D$ contains a curve of $C(A)$. Since $\length(A)$ is the minimal length of such a 
curve, $|\partial D|\geq \length(A)$. Moreover, $D$ cannot contain all the connected components of $\partial M$, which implies $|D\cap \partial M|\leq (b-1)a$. Hence, we have $\frac{|\partial D|}{|D\cap \partial M|}\geq \frac{\length(A)}{(b-1)a}$.
\end{enumerate}
We have considered all possible cases. To conclude, we observe that $\length(A)\leq a$ since the curves $\Sigma_i\times \{L\}$ belong to $C(A)$. Hence, we have $1\geq \frac{\length(A)}{a}\geq\frac{\length(A)}{(b-1)a}$. Since $\frac{|\partial D|}{|D\cap \partial M|}\geq\frac{\length(A)}{(b-1)a}$ or $\frac{|\partial D|}{|D\cap \partial M|}\geq\frac{L}{(b-1)a}$ for all possible $D$, we have $ h_2(A)\geq \frac{\min\{\length(A),L\}}{(b-1)a}$.
\end{proof}
We note that in higher dimensions, similar estimates cannot be obtained because in the second situation, the volume of $\partial D$ cannot be bounded below by $L$.
\begin{proof}[Proof of Theorem \ref{thm0}]
Theorem \ref{thm0} follows from Lemma \ref{lemma1} and Proposition \ref{propCheeg2}.
\end{proof}

The exponent of the geometric quantities involved in the estimate given in Theorem \ref{prop1} cannot be improved. This is obtained by showing that $\sigma_1$ and the lower bound are equivalent, in the sense that 
$\sigma_1$ goes to zero if and only if the lower bound goes to zero, for families of surfaces for which all geometric quantities involved in the lower bound except one are fixed. We recall that the Steklov eigenvalues of right cylinders can be computed.

\begin{prop} The Steklov eigenvalues of the right cylinder $S_R^1\times [-T,T]$, where $S_R^1$ denotes the circle of radius $R$, are
\[
0, \frac{1}{T}, \frac{k}{R}\tanh(\frac{k}{R}T)<\frac{k}{R}\coth(\frac{k}{R}T), \quad k\in \mathbb{N}^*.
\]
We note that if $\frac{T}{R}\geq \rho$, where $\rho\approx 1,19968$ is the positive root of $1=x\tanh(x)$, the first non-zero eigenvalue is $\frac{1}{T}$.
\label{vpcyl}
\end{prop}

\begin{exe}
Consider the sequence $\{M_n\}_{n\geq 1}$ where $M_n$ are right cylinders that have height $4\pi n$ and whose bases are unit circles. 
Since $2\pi n\geq \rho \text{ } \forall n \geq 1$, $\sigma_1(M_n)=\frac{1}{2\pi n}$. Hence, we have $\frac{4\pi}{|M_n|}=\frac{1}{2\pi n}=\sigma_1(M_n)\geq \frac{\length(M_n)^2}{2(b-1)a|M_n|}= \frac{\pi}{|M_n|}$.
\label{excyl}
\end{exe}
\begin{exe} Consider a surface $M_{\epsilon}$ with two boundary components of length $1$, having a cylindrical neighborhood of length $L$ and connected by a thin cylinder $C_{\epsilon}$ of circumference $\epsilon<1$ and of length $\frac{1}{\epsilon}$ (see Figure \ref{fig2}). Consider the function taking the value $-1$ on one side of $C_{\epsilon}$, $1$ on the other side, and extended continuously to a linear function on $C_{\epsilon}$, that is, on $C_{\epsilon}=S^1_{\frac{\epsilon}{2\pi}} \times [-\frac{1}{2\epsilon},\frac{1}{2\epsilon}]$, we have $f(s,t)=2\epsilon t$. Its Dirichlet energy is zero except on $C_{\epsilon}$ where it is
\[
\int_{C_{\epsilon}}|\nabla f|^2 dv_g=\int_0^{\epsilon}\int_{\frac{-1}{2\epsilon}}^{\frac{1}{2\epsilon}}4\epsilon^2dtds=\int_0^{\epsilon} 4 \epsilon ds=4\epsilon^2.
\]
Since the restriction of $f$ to the boundary is orthogonal to a constant function and $\int_{\partial M}f^2 dS_g=\int_{\partial M}1 dS_g=2$, we obtain
\[
\sigma_1(M)=\min\left\{R(u): u\in H^1(M), \int_{\partial M}u=0\right\}\leq R(f)=\frac{4\epsilon^2}{2}=2\epsilon^2.
\]
We note that if $L$ is small enough, the volume of $M_{\epsilon}$ satisfies $|M_{\epsilon}|\leq 2$. Hence, we have $2\length(M_{\epsilon})^2=2\epsilon^2\geq \sigma_1(M_{\epsilon})\geq\frac{\length(M_{\epsilon})^2}{2(b-1)a|M_{\epsilon}|}\geq \frac{ \length(M_{\epsilon})^2}{4}$.
\label{ex3}
\begin{figure}[H]
\includegraphics[scale=0.15]{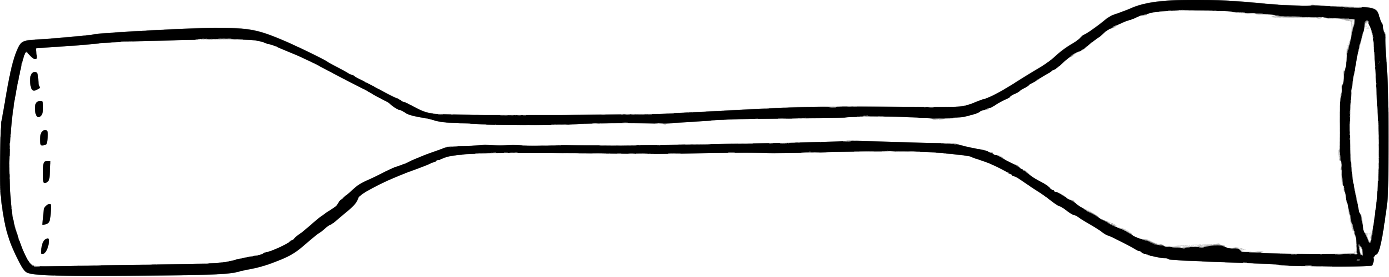}
\centering
\caption{A surface with two cylindrical boundary neighborhoods connected by a thin cylinder}
\label{fig2}
\end{figure}
 \end{exe}
Since we have shown that the exponent of $\min\{\length(M),1\}$ and $|M|$ cannot be improved, we can deduce that the exponent of $a$ must be $-1$ from the fact that the degree of homogeneity of the lower bound has to be consistent with the degree of homogeneity of $\sigma_1$. We conclude that, up to a constant, we cannot have a better lower bound for $\sigma_1$ depending on these geometric quantities (however, we may have different geometric quantities).

A lower bound is optimal if we can show that it goes to zero if and only if $\sigma_1$ goes to zero. Using the same strategy as in Example \ref{exSurfrev}, it is easy to construct a family of surfaces such that $\sigma_1$ is constant but the volume goes to infinity and therefore the lower bound given in Theorem \ref{prop1} tends to zero. This example shows that the volume of the manifold seems not to be an optimal quantity for estimating $\sigma_1$.  Theorem \ref{borneinf2} is an improvement of Theorem \ref{prop1} for surfaces whose Gaussian curvature is bounded below, which does not involve the volume of the manifold.

\begin{proof}[Proof of Theorem \ref{borneinf2}]
For $2\leq i\leq b$, we let $\gamma_i$ be a geodesic minimising the distance between $\Sigma_1$ and $\Sigma_i$. Around each $\gamma_i$, we consider the tube
\begin{multline*}
T_i=\{x\in M, \text{ there exists a geodesic }\xi \text{ of length } \\ l(\xi)<\inj_{\partial M}(M) \text{ from } x \text{ meeting } \gamma_i \text{ orthogonally}\}.
\end{multline*}
Since $\gamma_i$ meets $\partial M$ orthogonally, $T_i=\{x\in M, d(x,\gamma_i) < \inj_{\partial M}(M) \}$. We define ${A=\cup_{i=1}^{b}(\Sigma_i\times [0,\inj_{\partial M}(M)))\cup(\cup_{i=2}^{b}T_i)}$. We approximate the volume of $A$ by using a Bishop-Günther inequality for tubes (Theorem 8.16, point ii, in \cite{GrayTubes}). In the particular case of a tube $T$ of radius $r$ around a geodesic $\gamma$ in a surface whose Gaussian curvature is bounded from below by $\kappa<0$, this comparison result says that
\[
|T|\leq \frac{2l(\gamma)\sinh(\sqrt{-\kappa}r)}{\sqrt{-\kappa}}.
\]
\begin{figure}[H]
\includegraphics[scale=0.35]{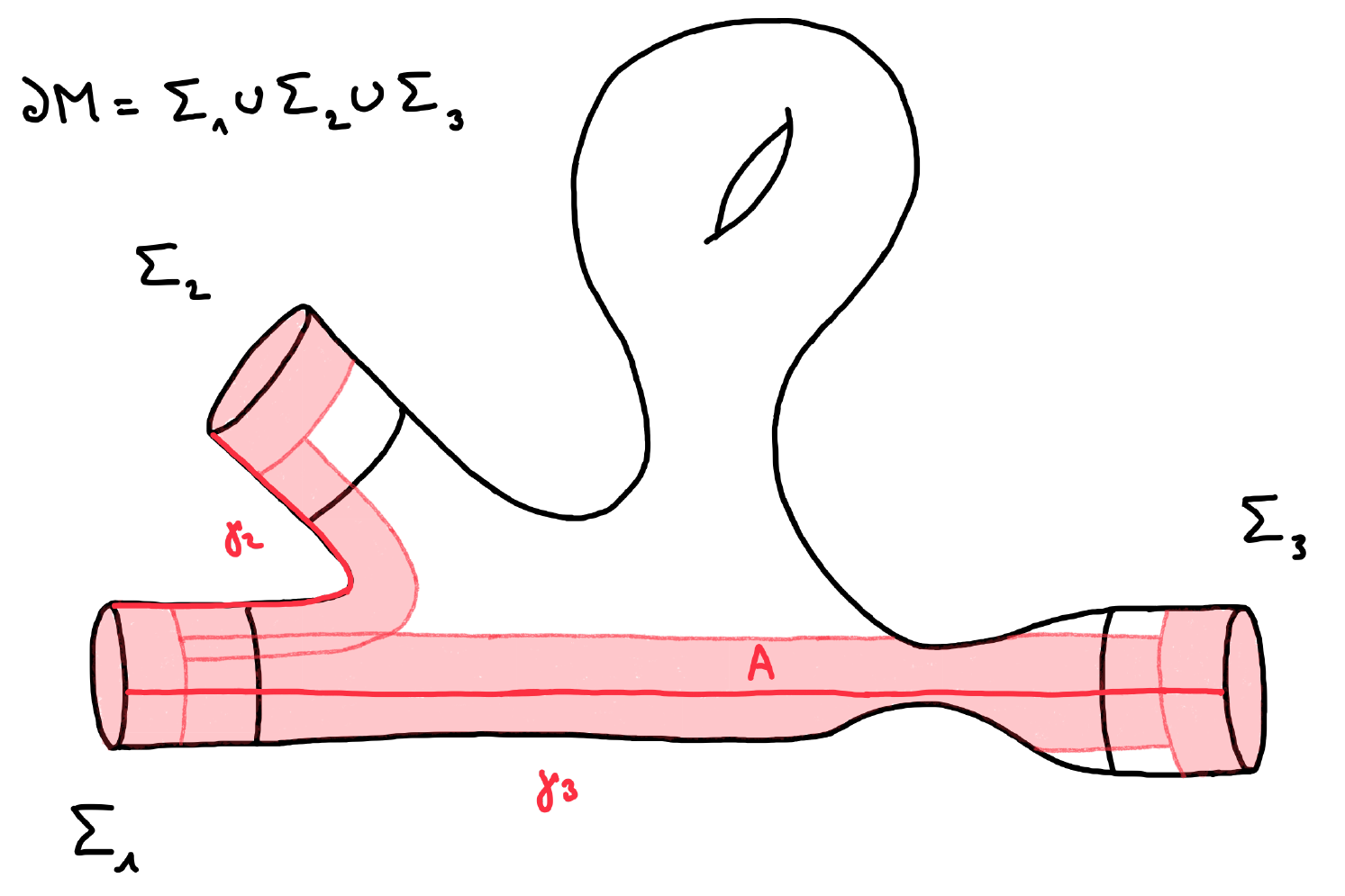}
\centering
\caption{The domain $A=\cup_{i=1}^{b}(\Sigma_i\times [0,\inj_{\partial M}(M)))\cup(\cup_{i=2}^{b}T_i)$}
\end{figure}
By applying this inequality to estimate the volume of the tubes $T_i$, we obtain
\begin{align*}
|A|&=|\cup_{i=1}^{b}(\Sigma_i\times [0,\inj_{\partial M}(M)))\cup(\cup_{i=2}^{b}T_i)|\\
&\leq ab\inj_{\partial M}(M)+\sum_{i=2}^{b}|(T_i)|\\
&\leq ab\inj_{\partial M}(M)+\sum_{i=2}^{b}\frac{2l(\gamma_i)\sinh(\sqrt{-\kappa}\inj_{\partial M}(M))}{\sqrt{-\kappa}}\\
&\leq ab\inj_{\partial M}(M)+\frac{2(b-1)\diam_M(\partial M)\sinh(\sqrt{-\kappa}\inj_{\partial M}(M))}{\sqrt{-\kappa}}.
\end{align*}
The set A can be approximated by smooth domains in the following way (for more details, see \cite{Daners}, Section 8.2). We define $V_n:=\{x\in A: d(x,\partial A)>\frac{1}{n}\}$ and consider $\phi_n$ a bump function for $\overline{V_{n}}$ supported in $V_{n+1}$ (on bump functions, see \cite{LeeSM}, Proposition 2.25). By Sard's Theorem, there exists $t_n\in (0,1)$ such that $E_n:=\{x\in M:\phi_n(x)>t_n\}$ is a smooth domain. Since $E_1\subset E_2\subset \dots$ and $\cup_{n\in\mathbb{N}^*}E_n=A$,  $|E_n|\rightarrow |A|$ as $n$ tends to infinity. Let $n_0$ be such that $\frac{1}{n_0}<\frac{\inj_{\partial M}(M)}{8}$. By taking $\tilde A=E_{n_0}$, we have that $\tilde A$ contains a cylindrical neighborhood of length $\frac{\inj_{\partial M}(M)}{2}$ of $\partial M$, $|A|\geq|\tilde A|$ and $\length(\tilde A)\geq \frac{\inj_{\partial M}(M)}{2}$. This last statement follows from the fact that
a curve $c$ which divides $\tilde A$ into two connected components, each containing at least one connected component of $\partial M$, must intersect a geodesic $\gamma_i$ at a point $x$ and cannot be contained in the ball $B_{\frac{\inj_{\partial M}(M)}{2}}(x)\subset \tilde A$.

Hence, by Theorem \ref{thm0}, we have
\[
\sigma_1(M)\geq \frac{\min\{\length(\tilde A),\frac{\inj_{\partial M}(M)}{2}\}^2}{2(b-1)a|\tilde A|}\geq\frac{\inj_{\partial M}(M)^2}{8(b-1)a|A|}.
\]
By combining the above inequality with the approximation of the volume of $A$, we obtain

\[
\sigma_1(M)\geq \frac{\inj_{\partial M}(M)^2}{8(b-1)a(ab\inj_{\partial M}(M)+\frac{2(b-1)\diam_M(\partial M)\sinh(\sqrt{-\kappa}\inj_{\partial M}(M))}{\sqrt{-\kappa}})}.
\]
Since by definition we have $\inj_{\partial M}(M)\leq1$, we obtain using the Taylor-Lagrange formula that 
\[
\frac{\sinh(\sqrt{-\kappa}\inj_{\partial M}(M))}{\sqrt{-\kappa}}\leq \cosh(\sqrt{-\kappa})\inj_{\partial M}(M).
\]
Hence, we have
\[
\sigma_1(M)\geq \frac{\inj_{\partial M}(M)}{8(b-1)a(ab+2(b-1)\diam_M(\partial M)\cosh(\sqrt{-\kappa}))}
\]
We note that this inequality is interesting in itself because it shows clearly how the different geometric quantities affect the lower bound. If we assume that $a\leq \diam_M(\partial M)$, we obtain
\begin{align*}
\sigma_1(M)&\geq \frac{\inj_{\partial M}(M)}{8(b-1)a(\diam_M(\partial M)b+2(b-1)\diam_M(\partial M)\cosh(\sqrt{-\kappa}))}\\
& \geq\frac{\inj_{\partial M}(M)}{16b^2\cosh(\sqrt{-\kappa})a\diam_M(\partial M)}\\
&=C(\kappa,b)\frac{\inj_{\partial M}(M)}{a \diam_M(\partial M)},
\end{align*}
where $C(\kappa,b)=\frac{1}{16b^2\cosh(\sqrt{-\kappa})}$.
\end{proof}

\begin{remark}
The exponent of the geometric quantities involved in Theorem \ref{borneinf2} cannot be improved. To show this, we proceed in the same way as for Theorem \ref{prop1}. We first observe that the exponent of the diameter of the boundary cannot be improved because the family of right cylinders of fixed base and growing height $\{M_n\}_{n\geq1}$ of Example \ref{excyl} satisfy $\sigma_1(M_n)=\frac{1}{2\pi n}=\frac{2}{\diam_M(\partial M_n)}$ and, except the diameter of the boundary, the quantities involved in the lower bound are fixed. We consider now the family of right cylinders $\{M_a\}_{a\geq1}$ of height $2$ and growing base of length $a$. The only quantities involved in the lower bound that are changing are the diameter of the boundary and the length of the boundary. We have $\sigma_1(M_a)\leq \frac{4\pi}{a^2}$, which shows that the exponent of $a$ is also optimal because we have already showed that the exponent of the diameter of the boundary cannot be improved. For obtaining that the exponent of the injectivity radius if optimal, we note that in Example \ref{ex3} we can construct the surfaces $M_{\epsilon}$ so that their Gaussian curvature is bounded from below. This can be done by joining the inner cylinder and the two cylindrical neighborhoods of the boundary by a cylinder of constant Gaussian curvature equal to $-1$ and smoothing the joints. Here, the only quantities involved in the lower bound that are changing are the injectivity radius and the diameter of the boundary. From Example \ref{ex3}, we have $\sigma_1(M_{\epsilon})\leq 2\epsilon^2\leq 8\inj_{\partial M_{\epsilon}}(M_{\epsilon})^2$. On the other hand, by construction, $\diam_M(\partial M_\epsilon)$ is of the same order as $\frac{1}{\inj_{\partial M_{\epsilon}}(M_{\epsilon})}$ as $\epsilon$ goes to zero. This implies that there exists a constant $c$ such that $\sigma_1(M_{\epsilon})\geq  c \inj_{\partial M_{\epsilon}}(M_{\epsilon})^2$. Hence the exponent of the injectivity radius is optimal because we have already showed that the exponent of the diameter of the boundary cannot be improved.
\label{remopt}
\end{remark}

We remark that if $a$ goes to zero, $\length(M)$ and $\inj_{\partial M}(M)$ also go to zero. Therefore, Theorems \ref{prop1} and \ref{borneinf2} do not say that $\sigma_1$ goes to infinity as $a$ goes to zero, which is not true, as shown by the following example. We consider the sequence of right cylinders $\{S^1_{\frac{1}{n}}\times [-1,1]\}_{n\geq 1}$. Proposition \ref{vpcyl} shows that if $n\geq2$, $\sigma_1=1$. By taking the sequence $\{S^1_{\frac{1}{n}}\times [-n,n]\}_{n\geq 1}$, we even have that $\sigma_1$ tends to zero as the length of the boundary tends to zero. This is in contrast to the case of surfaces with one cylindrical boundary component where Lemma \ref{compcyl} shows that $\sigma_1$ goes to infinity as the length of the boundary goes to zero.

\end{subsection}

\begin{subsection}{Geometric bounds on the low Steklov eigenvalues of a compact hyperbolic surface with geodesic boundary}

A compact hyperbolic surface of signature $(g,b)$ is a compact 2-dimensional Riemannian manifold of constant Gaussian curvature equal to $-1$ with genus $g$ and a geodesic boundary having $b$ connected components. An important property of hyperbolic surfaces is that they are isometric to a warped product around simple closed geodesics.

\begin{prop} Let $M$ be a closed hyperbolic surface of genus $g\geq2$ and let $\gamma_1,\dots,\gamma_m$ be pairwise disjoint simple closed geodesics on $M$. Then
$m\leq 3g-3$ and there exist simple closed geodesics $\gamma_{m+1},\dots,\gamma_{3g-3}$ which, together with $\gamma_1,\dots,\gamma_m$, decompose $M$ into surfaces of signature $(0,3)$. Moreover, the collars 
\[
K(\gamma_i)=\{p\in M, dist(p,\gamma_i)\leq w(\gamma_i)\}
\]
where
\[
w(\gamma_i)=\arcsinh(\frac{1}{\sinh(\frac{1}{2}l(\gamma_i))})
\]
are pairwise disjoint and each collar $K(\gamma_i)$ is isometric to the cylinder $S^1\times [-w(\gamma_i),w(\gamma_i)]$ with the metric $g(s,t)=\frac{l^2(\gamma_i)\cosh^2(t)}{(2\pi)^2}g_{S^1}(s)+dt^2$ where $g_{S^1}$ is the canonical metric on $S^1$.
\label{Collar}
\end{prop}
For a proof of this result, we refer to  \cite{BuserBook92}, Theorem 4.1.1. A direct consequence is that a hyperbolic surface with geodesic boundary has a boundary neighborhood which is isometric to a union of disjoint warped products. This implies the following approximation of the Steklov eigenvalues.

\begin{lemma}
Let $M$ be a hyperbolic surface with $b\geq2$ geodesic boundary components of length $a$. Then, the Steklov eigenvalues $\sigma_k$ of $M$ satisfy
\begin{align*}
0\leq\sigma_k\leq \frac{1}{\arctan(\frac{1}{\sinh{\frac{a}{2}}})}
\end{align*}
if $k<b$, and
\begin{align*}
\frac{2\pi j}{a}\tanh(\frac{2\pi j}{a}\arctan(\frac{1}{\sinh(\frac{a}{2})}))\leq\sigma_k\leq \frac{2\pi j}{a}\coth(\frac{2\pi j}{a}\arctan(\frac{1}{\sinh(\frac{a}{2})}))
\end{align*}
if $(2j-1)b\leq k< (2j+1)b$, where $j\in\mathbb{N}^*$.
\label{compSNhyp}
\end{lemma}
\begin{proof} 
Let $\Sigma_1,...,\Sigma_b$ be the $b$ boundary components, where $l(\Sigma_1)=...=l(\Sigma_b)=a$. By gluing a hyperbolic surface of signature $(1,1)$ to each boundary component, we obtain a closed hyperbolic surface  of genus  $g\geq 2$. Theorem  \ref{Collar} says that the collars $K(\Sigma_i)=\{p\in M, dist(p,\Sigma_i)\leq w(\Sigma_i)\}$, where $w(\Sigma_i)=\arcsinh(\frac{1}{\sinh(\frac{a}{2})})$, are disjoint and isometric to cylinders $S^1\times [0,w(\gamma_i)]$ with the metric $g(s,t)=\frac{a^2\cosh^2(t)}{(2\pi)^2}g_{S^1}(s)+dt^2$. Let $A=\cup_i K(\Sigma_i)$ be the union of these collars. We consider the mixed Steklov-Neumann and Steklov-Dirichlet problems on $A$. From Equation \ref{compvp}, we have
\[
\sigma_i^N(A)\leq\sigma_i(M)\leq\sigma_i^D(A).
\]
Since the $K(\Sigma_i)$ are warped products, the eigenvalues of these mixed problems can be explicitly calculated. The result is obtained by replacing $\sigma_i^N(A)$ and $\sigma_i^D(A)$ by the their exact value in the previous equation.

\end{proof}

A classical result due to L. Bers says that every closed hyperbolic surface of genus $g\geq 2$ admits a decomposition into surfaces of signature $(0,3)$ such that the length of the separating geodesics is controlled by a constant depending on the genus. We give a statement of this result due to P. Buser  (see \cite{BuserBook92}, Theorem 5.2.3) which is convenient to deduce an analog result for surfaces with geodesic boundary of controlled length.

\begin{prop}
Let $M$ be a closed hyperbolic surface of genus $g\geq 2$ and let $\gamma_1,...,\gamma_m$ be the set of all distinct simple closed geodesics of length $l\leq 2\arcsinh(1)$. This system is extendable to a partition $\gamma_1,\dots,\gamma_{3g-3}$ satisfying
\[
l(\gamma_k)\leq 4k\log(\frac{8\pi(g-1)}{k}), \quad k=1,...,3g-3.
\]
\end{prop}

\begin{cor} There exists a constant $L_{g+b}$, depending only on $g$ and $b$, such that every hyperbolic surface $M$ of genus $g$ with $b\geq2$ geodesic boundary components of length $l\leq 2\arcsinh(1)$ can be decomposed into surfaces of signature $(0,3)$ by simple closed geodesics $\gamma_1,...,\gamma_{3g-3+b}$ satisfying
\[
l(\gamma_i)\leq L_{g+b}, \quad i=1,...,3g-3+b.
\]
\label{bersbord}
\end{cor}
\begin{proof}
Let $\gamma_1,...,\gamma_b$ be the geodesic boundary components of $M$. By gluing a hyperbolic surface of signature $(1,1)$ to each boundary component, we obtain a closed hyperbolic surface $M'$ of genus $g+b\geq 2$ and $\gamma_1,\dots,\gamma_b$ are closed geodesics of $M'$ of length $l\leq2\arcsinh(1)$. We add to this set all distinct simple closed geodesics on $M'$ of length $l\leq 2\arcsinh(1)$. From Bers'Theorem, the resulting set $\gamma_1,\dots,\gamma_m$ can be extended to a partition $\gamma_1,...,\gamma_{3(g+b)-3}$ of simple closed geodesics satisfying $l(\gamma_k)\leq 4k\log(\frac{8\pi(g+b-1)}{k}) \text{ for }k=1,...,3(g+b)-3$. In particular, there exists a constant $L_{g+b}=4(3(g+b)-3)\log(\frac{8\pi(g+b-1)}{3(g+b)-3})$ such that $l(\gamma_k)\leq L_{g+b}$ for $k=1,...,3(g+b)-3$. Among this family of geodesics, we have the $b$ geodesics $\gamma_1,...,\gamma_b$ of the boundary of $M$ and we also have $b$ simple closed geodesics that divide the surfaces of signature $(1,1)$ glued at each boundary to make them surfaces of signature $(0,3)$. The $3g-3+b$ remaining geodesics decompose $M$ into surfaces of signature $(0,3)$ and their length is bounded by $L_{g+b}$.
\end{proof}

We are now able to give the proof of Theorem \ref{thmhyp}. The strategy is the same as the strategy used in \cite{SWY} for obtaining a result for Laplace eigenvalues.

\begin{proof}[Proof of Theorem \ref{thmhyp}]
\item
\begin{paragraph}{Step 1: $\length_n\leq\beta_1$ where $\beta_1$ is a constant depending only on $g$ and $b$.}
From Corollary \ref{bersbord} there exists a family of simple closed geodesics of length $l\leq L_{g+b}$, dividing $M$ into $3g-3+b$ surfaces of signature $(0,3)$. Since we assume that $M$ is connected, each of this surfaces of signature $(0,3)$ contains at most two components of $\partial M$. Hence, by choosing a subset of these geodesics, we can obtain for $1\leq n<\lceil \frac{b}{2}\rceil$ a division of $M$ into $n+1$ connected components, each one containing at least one component of $\partial M$. Let $\gamma$ denote the curve consisting of the union of these geodesics. Because $\gamma$ consists of at most $3g-3+b$ geodesics of length $l\leq L_{g+b}$, there exists a constant $\beta_1$, depending only on $g$ and $b$ and such that $l(\gamma)\leq \beta_1$. Let $c\in C_n(M)$ be a curve satisfying $l(c)=\length_n$. Since $\gamma\in C_n(M)$, we have $\length_n \leq l(\gamma) \leq \beta_1$. For $\lceil \frac{b}{2}\rceil\leq n<b$, the assumption says that there exists $c\in C_n(M)$ such that each simple closed geodesic of $c$ is of length $l\leq L_{g+b}$. Since $c$ consists of at most $3g-3+b$ geodesics, we have $\length_n \leq l(c) \leq \beta_1$.
\end{paragraph}

\begin{paragraph}{Step 2: $\sigma_n\leq C_2 \frac{\length_n}{a}$.}
If $\length_n>1$, we obtain from the combination of Lemma \ref{compSNhyp} and the hypothesis that $a\leq2\arcsinh(1)$ that $\sigma_n\leq \frac{1}{\arctan(\frac{1}{\sinh{\frac{a}{2}}})}<\frac{8\arcsinh(1)\length_n}{\pi a}$. Now assume that $\length_n\leq1$.
Let $c\in C_n(M)$ be the curve from step 1 satisfying $l(c)=\length_n$. This curve decompose $M$ into $n+1$ connected components $M_1,\dots,M_{n+1}$, each one containing at least one boundary component. We suppose $c=\gamma_1\cup\dots\cup \gamma_p$ where the $\gamma_i$ are simple closed geodesics on $M$. From Proposition \ref{Collar}, we know that there exist disjoint collars $K(\gamma_1),\dots,K(\gamma_p)$ about the geodesics $\gamma_1,\dots, \gamma_p$. If $K_j\cap M_i\not=\emptyset$, $K_j\cap \overline{M_i}$ is isometric to $ 
S^1\times [0,w(\gamma_j)]$ with the metric $g(s,t)=\frac{l^2(\gamma_i)\cosh^2(t)}{(2\pi)^2}g_{S^1}(s)+dt^2$, and $K_j\cap \partial M_i$ corresponds to $S^1\times \{0\}$. The upper bound is obtained by using test functions. We define
\[
\phi_i(x)=
\begin{cases}1 & \text{if } x \in M_i\setminus \cup_{j=1}^{p}K_j;\\
\phi_{i,j}(x)& \text{if } x\in M_i\cap K_j \text{ for a }j=1,\dots,p;\\
0 &\text{ otherwise;}
\end{cases}
\]
and
\begin{align*}
\phi_{i,j}:S^1\times [0,w(\gamma_j)]& \rightarrow \mathbb{R}\\
(s,t)&\mapsto \frac{\arctan(\sinh(t))}{\arctan(\frac{1}{\sinh(\frac{l(\gamma_i)}{2})})}.
\end{align*}
The Dirichlet energy of this function on the half-collar $M_i\cap K_j$ is
\[
\int_{S^1\times [0,w(\gamma_i)]}|\nabla\phi_{i,j} |^{2}dv=\frac{l(\gamma_j)}{\arctan(\frac{1}{\sinh(\frac{l(\gamma_j)}{2})})}.
\]
Let $E$ be the set of indices $j$ such that $M_i\cap K_j\not = \emptyset$. The total Dirichlet energy of $\phi_i$ satisfies
\begin{align*}
\int_{M_i}|\nabla\phi_i |^{2}dv&=\sum_{j\in E}\frac{l(\gamma_j)}{\arctan(\frac{1}{\sinh(\frac{l(\gamma_j)}{2})})}\\
&\leq\frac{\sum_{j\in E}l(\gamma_j)}{\arctan(\frac{1}{\sinh(\frac{\sum_{j=1}^{r}{l(\gamma_j)}}{2})})}\\
&\leq\frac{\length_n}{\arctan(\frac{1}{\sinh(\frac{\length_n}{2})})}.
\end{align*}
We also have
\[
\int_{\partial M_i} \phi_i^2 dS=m\times a\geq a,
\]
where $m$ is the number of boundary components included in $M_i$. Hence the Rayleigh quotient of $\phi_i$ satisfy
\[
R(\phi_i)\leq \frac{\length_n}{a \arctan(\frac{1}{\sinh(\frac{\length_n}{2})})}.
\]
Since $\length_n\leq 1$, we have $\frac{1}{\arctan(\frac{1}{\sinh(\frac{\length_n}{2})})}<\frac{1}{\arctan(\frac{1}{\sinh(\frac{1}{2})})}=:\beta_2$ and
\[
R(\phi_i)\leq \beta_2\frac{\length_n}{a}.
\]
Let $V$ be the linear span of $\phi_1,\dots,\phi_{n+1}$ in $H^1(M)$. Since the functions $\phi_i$ have disjoint support, $V$ is an $(n+1)$-dimensional vector space and we have
\[
\max\{R(u), u\in V\}=\max \{R(\phi_1),...,R(\phi_{n+1})\}.
\]
\begin{sloppypar}
Since $R(\phi_i)\leq \beta_2\frac{\length_n}{a}$ for $i=1,...,n+1$, we have $\max \{R(\phi_1),...,R(\phi_{n+1})\} \leq \beta_2\frac{\length_n}{a}$. Using the variational characterization $\sigma_n(M)=\min_{V\in V_k}\max_{0\not = u\in V} R(u)$, where $V_k$ is the set of all $(k+1)$-dimensional linear subspace of $H^1(M)$, we obtain
\end{sloppypar}
\[
\sigma_n(M)\leq \beta_2\frac{\length_n}{a}.
\]
Because we have obtained the desired result both when $\length_n >1$ and when $\length_n\leq1$, we have
\[
\sigma_n(M)\leq C_2\frac{\length_n}{a},
\]
where $C_2=\max\{\frac{8\arcsinh(1)}{\pi},\beta_2\}$ is a universal constant.

\end{paragraph}

 \begin{paragraph}{Step 3: $C_1\length_n^2\leq \sigma_n$.}
Since $l(c)=\length_n$, one of the $p$ components $\gamma_i$ of $c$ must satisfy $l(\gamma_i)\geq \frac{\length_n}{p}$; we call it $\gamma_{\max}$. The geodesic $\gamma_{\max}$ is contained in the boundary of two sets $M_j$ and $M_k$. We let $\Omega_1=M_j\cup M_k\cup(\partial M_j\cap\partial M_k)$ and $\Omega_2,\dots,\Omega_n$ be the remaining $M_i$. Let $A=\cup_{i=1}^{n}\Omega_i$. On each $\Omega_i$, we consider the mixed Steklov-Neumann problem with Steklov condition on $\Omega_i\cap \partial M$ and Neumann condition on $\partial \Omega_i$. Since the $\Omega_i$ are disjoint, we have
 \[
 \sigma_{n}^N(A)=\min\{\sigma_1^N(\Omega_1),...,\sigma_1^N(\Omega_{n})\}
 \]
and since $A$ contains all boundary components of $M$, we have
\[
\sigma_k(M)\geq \sigma_k^N(A).
\]
Therefore, the proof will be finished if we can show that $\sigma_1^N(\Omega_i)\geq\alpha_1 \length_n^2$ for $i=1,\dots,n$.
If $\Omega_i$ contains only one boundary component $\Sigma_i$, we consider the mixed Steklov-Neumann problem on the half-collar $K(\Sigma_i)$. By comparing the Rayleigh quotients, we see that $\sigma_1^N(\Omega_i)\geq\sigma_1^N(K(\Sigma_i))$. We have already mentioned that a calculation shows that $\sigma_1^N(K(\Sigma_i))=\frac{2\pi}{a}\tanh(\frac{2\pi}{a}\arctan(\frac{1}{\sinh(\frac{a}{2})}))$. Since $a\leq2\arcsinh(1)$, by letting $\beta_3=\frac{\pi}{\arcsinh(1)}\tanh(\frac{\pi}{\arcsinh(1)\arctan(1)})$, we obtain $\sigma_1^N(K)\geq\beta_3$.

If $\Omega_i$ contains several boundary components, we obtain the result by estimating the constants $h_1(\Omega_i)$ et $ h_2(\Omega_i)$ and using Proposition \ref{propCheeg1} and Remark \ref{remMixed}.

\begin{subparagraph}{Estimation of $ h_1(\Omega_i)$}
We recall that
\[
h_1(\Omega_i):=\inf \frac{|\partial D|}{|D|}
\]
where the infimum if taken among all domains $D$ of $\Omega_i$ satisfying $|D|\leq\frac{|\Omega_i|}{2}$, $D\cap \partial M \not =\emptyset$, and such that $M\setminus D$ is also connected and intersects $\partial M$. Given such a domain $D$, we have the following possibilities that are illustrated in Figure \ref{FigIII.4}.

\begin{figure}[H]
\includegraphics[scale=0.23]{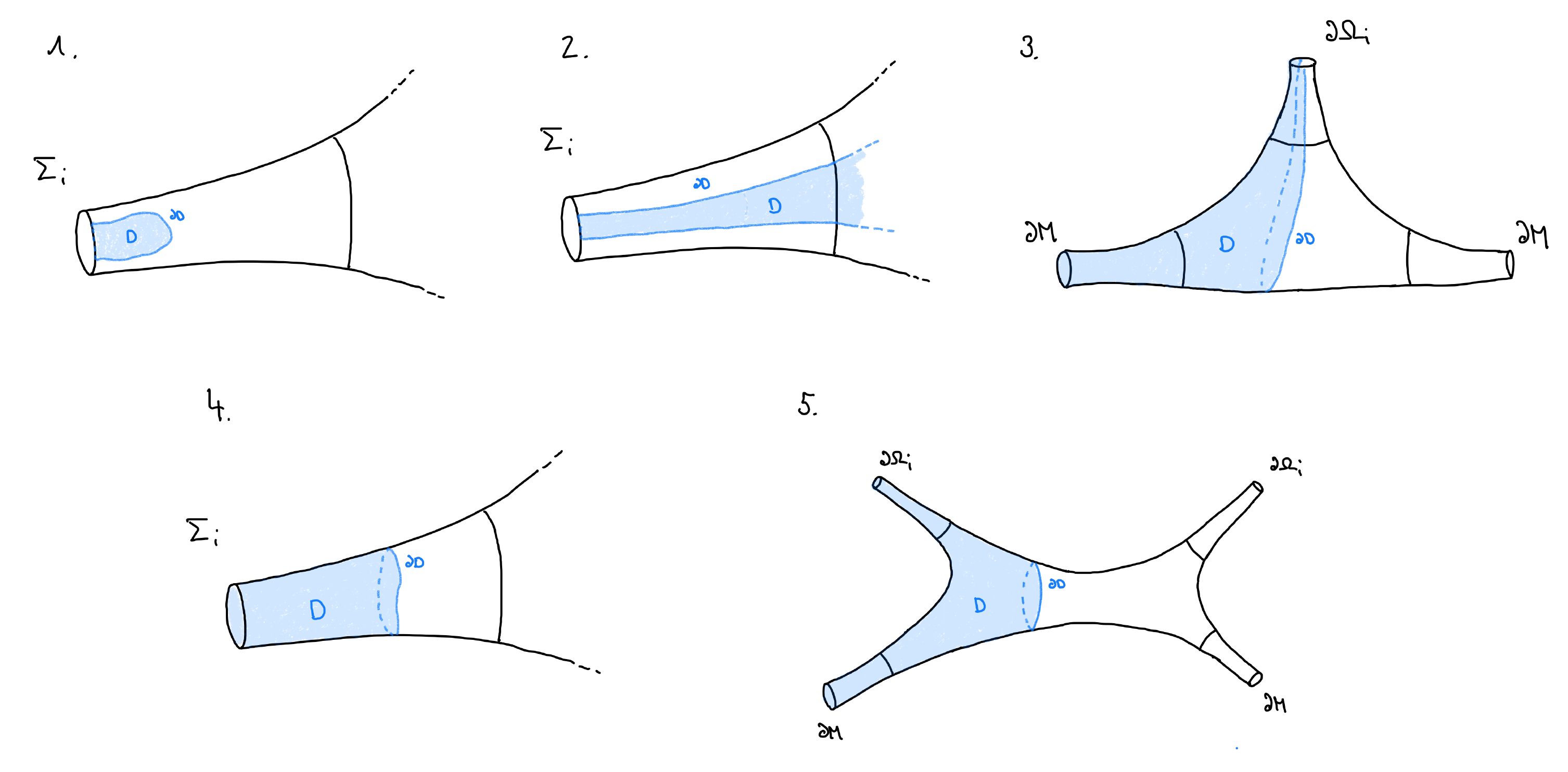}
\centering
\caption{A schematic representation of possible configurations of $D$ in $\Omega_i$ corresponding to each of the five cases}
\label{FigIII.4}
\end{figure}

\begin{enumerate}

\item $\partial D$ intersects a component $\Sigma_i$ of $\partial M$ and $D$ is contained in the collar neighborhood $K(\Sigma_i)$. From the isoperimetric inequality for simply connected domains in the hyperbolic plane we know that $|D|\leq |\partial D|$. So we have $\frac{|\partial D|}{|D|}\geq \frac{|\partial D|}{ |\partial D|}=1$.
\item $\partial D$ intersects a boundary component $\Sigma_i$ but $D$ is not contained in $K(\Sigma_i)$. Since $w(\Sigma_i)\geq\arcsinh(1)$, we have $|\partial D|\geq w(\Sigma_i)\geq\arcsinh(1)$. Therefore $\frac{|\partial D|}{|D|}\geq\frac{\arcsinh(1)}{|M|}=\frac{\arcsinh(1)}{2\pi(2g-2+b)}=:\beta_4$. We see that $\beta_4$ only depends on $g$ and $b$.
\item $\partial D$ intersects a boundary geodesic $\gamma_i$ of $\partial \Omega_i$. Since both $D$ and $\Omega_i\setminus D$ have to intersect $\partial M$, $\partial D$ cannot be contained in $K(\gamma_i)$. Since $l(\gamma_i)\leq L_{g+b}$, $|\partial D|\geq w(\gamma_i)\geq\beta_5$ where $\beta_5$ is a constant depending only on $g$ and $b$. Thus $\frac{|\partial D|}{|D|}\geq\frac{\beta_5}{|M|}=\frac{\beta_5}{2\pi(2g-2+b)}=:\beta_6$.
\item $\partial D$ does not intersect $\partial M$ and a component $\Gamma$ of $\partial D$ is freely homotopic to a boundary component $\Sigma_i$. In this case $\Gamma \cup \Sigma_i$ bounds an annulus and since the Gaussian curvature is negative and $\Sigma_i$ is a geodesic, $|\Gamma|\geq l(\Sigma_i)$. If $|\Gamma|\geq 1$, we have $\frac{|\partial D|}{|D|}\geq \frac{1}{|M|}=\frac{1}{2\pi(2g-2+b)}$. If $|\Gamma|<1$, we deduce that $\frac{|\partial D|}{|D|}\geq\frac{1}{10}$ from  the isoperimetric inequality given in Theorem 3 of \cite{YauIsopcon} and the fact that $|\Gamma|\geq l(\Sigma_i)$. Thus we have $\frac{|\partial D|}{|D|}\geq \min\{\frac{1}{2\pi(2g-2+b)},\frac{1}{10}\}=\frac{1}{2\pi(2g-2+b)}\geq \beta_4$.

\item $\partial D$ does not intersect $\partial M$ and none of its components is freely homotopic to a boundary component. We note that each component of $\partial D$ is freely homotopic to a simple closed geodesic of $\Omega_i$. Let $\Gamma$ be the union of these geodesics. We have $|\partial D|\geq |\Gamma|$. $\Gamma$ divides $\Omega_i$ into two connected components, each of them containing at least one connected component of $\partial M$. We recall that the geodesics $\gamma_1,...,\gamma_p$ divide $M$ into $n+1$ regions and that a subset of these geodesics divides $M$ into $n$ regions $\Omega_1,...,\Omega_n$. Let $\tilde \gamma$ be the union of the geodesics that do not belong to this subset. We have $\gamma_{max}\in \tilde \gamma$. If $|\Gamma|$ were smaller than $l(\tilde \gamma)$ there would be a family of geodesics of $M$, dividing $M$ into $n+1$ regions and their total length would be smaller than $\length_n$, which is a contradiction. Hence we have $|\Gamma|\geq l(\tilde\gamma)\geq l(\gamma_{\max})\geq \frac{\length_n}{p}$ and since $3g-3+b$ is the maximal number of these geodesics, $|\Gamma|\geq \frac{\length_n}{3g-3+b}$. Therefore,  $\frac{|\partial D|}{|D|}\geq \frac{|\Gamma|}{|M|}\geq \frac{\length_n}{(3g-3+b)2\pi(2g-2+b)}=\frac{\length_n}{\beta_7}$ where $\beta_7$ is a constant depending only on $g$ and $b$.
\end{enumerate}
Since we have considered all possibilities for $\partial D$, we have
\[
h_1(\Omega_i)\geq \min\{1,\beta_4,\beta_6, \frac{\length_n}{\beta_7}\}.
\]
Since $\length_n\leq \beta_1$, $h_1(\Omega_i)\geq\beta_8 \length_n$ where $\beta_8=\min\{\beta_1^{-1},\beta_4\beta_1^{-1},\beta_6\beta_1^{-1},\frac{1}{\beta_7}\}$ is a constant depending only on $g$ and $b$.

\end{subparagraph}

\begin{subparagraph}{Estimation of $h_2(\Omega_i)$}
We recall
\[
h_2(\Omega_i):=\inf \frac{|\partial D|}{|D\cap \partial M|}
\]
where the infimum if taken among all domains $D$ of $\Omega_i$ satisfying $|D|\leq\frac{|\Omega_i|}{2}$, $D\cap \partial M \not =\emptyset$, and such that $M\setminus D$ is also connected and intersects $\partial M$. Given such a domain $D$, we have the following possibilities that are illustrated in Figure \ref{FigIII.4}.

\begin{enumerate}

\item $\partial D$ intersects a component $\Sigma_i$ of $\partial M$ and $D$ is contained in the collar neighborhood $K(\Sigma_i)$. Since the Gaussian curvature is negative and $\Sigma_i$ is a geodesic, $|\partial D|\geq l(D\cap\Sigma_i)$. Thus, we have $\frac{|\partial D|}{|D\cap \partial M|}\geq \frac{|D\cap \Sigma_i|}{ |D\cap\Sigma_i|}=1$.

\item 
$\partial D$ intersects a boundary component $\Sigma_i$ but $D$ is not contained in $K(\Sigma_i)$. Since $l(\Sigma_i)\leq2\arcsinh(1)$, we have $|\partial D|\geq w(\Sigma_i)\geq\arcsinh(1)$, which implies$\frac{|\partial D|}{|D\cap \partial M|}\geq\frac{\arcsinh(1)}{ba}\geq\frac{1}{2b}$.

\item 
$\partial D$ intersects a boundary geodesic $\gamma_i$ of $\partial \Omega_i$. Since both $D$ and $\Omega_i\setminus D$ have to intersect $\partial M$, $\partial D$ cannot be contained in $K(\gamma_i)$. Since $l(\gamma_i)\leq L_{g+b}$, $|\partial D|\geq w(\gamma_i)\geq\beta_5$ where $\beta_5$ is a constant depending only on $g$ and $b$. Thus $\frac{|\partial D|}{|D\cap\partial M|}\geq\frac{\beta_5}{ba}=\frac{\beta_5}{2\arcsinh(1)b}=:\beta_9$.

\item 
$\partial D$ does not intersect $\partial M$ and a component $\Gamma$ of $\partial D$ is freely homotopic to a boundary component $\Sigma_i$. Since the Gaussian curvature is negative and $\Sigma_i$ is a geodesic, $|\partial D|\geq l(\Sigma_i)$. Therefore, we have $\frac{|\partial D|}{|D\cap \partial M|}\geq \frac{|\Sigma_i|}{ |\Sigma_i|}=1$.

\item 
$\partial D$ does not intersect $\partial M$ and none of its components is freely homotopic to a boundary component. We note that each component of $\partial D$ is freely homotopic to a simple closed geodesic of $\Omega_i$. Let $\Gamma$ be the union of these geodesics. We have $|\partial D|\geq |\Gamma|$. $\Gamma$ divides $\Omega_i$ into two connected components, each of them containing at least one connected component of $\partial M$. We recall that the geodesics $\gamma_1,...,\gamma_p$ divide $M$ into $n+1$ regions and that a subset of these geodesics divides $M$ into $n$ regions $\Omega_1,...,\Omega_n$. Let $\tilde \gamma$ be the union of the geodesics that do not belong to this subset. We have $\gamma_{max}\in \tilde \gamma$. If $|\Gamma|$ were smaller than $l(\tilde \gamma)$ there would be a family of geodesics of $M$, dividing $M$ into $n+1$ regions and their total length would be smaller than $\length_n$, which is a contradiction. Hence we have $|\Gamma|\geq l(\tilde\gamma)\geq l(\gamma_{\max})\geq \frac{\length_n}{p}$ and since $3g-3+b$ is the maximal number of these geodesics, $|\Gamma|\geq \frac{\length_n}{3g-3+b}$. Therefore, $\frac{|\partial D|}{|D\cap \partial M|}\geq \frac{|\Gamma|}{ab}\geq \frac{\length_n}{(3g-3+b)(2\arcsinh(1)b)}=\beta_{10} \length_n$  and $\beta_{10}$ is a constant depending only on $g$ and $b$. 
\end{enumerate}

Since we have considered all possibilities for $\partial D$, we have
\[
h_2(\Omega_i)\geq \min\{1,\frac{1}{2b},\beta_9,\beta_{10} \length_n\}.
\]
Since $\length_n\leq \beta_1$, $h_2(\Omega_i)\geq\beta_{11} \length_n$, where $\beta_{11}:=\min\{\beta_1^{-1},\frac{1}{2b}\beta_1^{-1},\beta_9\beta_1^{-1},\beta_{10}\}$ is a constant depending only on $g$ and $b$.

\end{subparagraph}
If $\Omega_i$ has several boundary components, we have shown that $\sigma_1^N(\Omega_i)\geq\frac {h_1(\Omega_i) h_2(\Omega_i)}{4}\geq \frac{\beta_8\beta_{11} \length_n^2}{4}=\beta_{12} \length_n^2$ where $\beta_{12}$ is a constant depending only on $g$ and $b$.

We conclude that $\sigma_1^N(\Omega_i)\geq \min\{\beta_3, \beta_{12} \length_n^2\}$. Since $\length_n\leq \beta_1$, $\sigma_1^N(\Omega_i)\geq \beta_{13} \length_n^2$
where $\beta_{13}=\min\{\beta_3\beta_1^{-2},\beta_{12}\}$.
Since it is true for all $\Omega_i$, we obtain
\[
\sigma_{n}(M)\geq\min\{\sigma_1^N(\Omega_1),...,\sigma_1^N(\Omega_n)\}\geq \beta_{13} \length_n^2,
\]
where $\beta_{13}$ is a constant depending only on $g$ and $b$.
\end{paragraph}
\end{proof}

\begin{remark}
We have seen that the presence of the area of $M$ in the denominator of the lower bound of Theorem \ref{prop1} can make this estimate inaaccurate. In Theorem \ref{thmhyp} the weight of the area of $M$ is hidden in the constant since it depends only on the signature of the hyperbolic surface.
\end{remark}

\end{subsection}

\end{section}

\begin{paragraph}{Acknowledgements}
I would like to thank the anonymous referee for the helpful comments and suggestions.
\end{paragraph}

\bibliography{biblio}{}

\begin{thebibliography}{10}

\bibitem{brissonMT}
Jade Brisson.
\newblock Probl\`emes isop\'erim\'etriques et isospectralit\'e pour le
  probl\`eme de {S}teklov.
\newblock Master's thesis, Universit\'e Laval, Qu\'ebec, Canada, 2019.

\bibitem{Buser77}
Peter Buser.
\newblock \"{U}ber den ersten {E}igenwert des {L}aplace-{O}perators auf
  kompakten {F}l\"{a}chen.
\newblock {\em Comment. Math. Helv.}, 54(3):477--493, 1979.

\bibitem{BuserHawaii}
Peter Buser.
\newblock On {C}heeger's inequality {$\lambda \sb{1}\geq h\sp{2}/4$}.
\newblock In {\em Geometry of the {L}aplace operator ({P}roc. {S}ympos. {P}ure
  {M}ath., {U}niv. {H}awaii, {H}onolulu, {H}awaii, 1979)}, Proc. Sympos. Pure
  Math., XXXVI, pages 29--77. Amer. Math. Soc., Providence, R.I., 1980.

\bibitem{BuserBook92}
Peter Buser.
\newblock {\em Geometry and spectra of compact {R}iemann surfaces}.
\newblock Modern Birkh\"{a}user Classics. Birkh\"{a}user Boston, Ltd., Boston,
  MA, 2010.
\newblock Reprint of the 1992 edition.

\bibitem{Cheeger1970}
Jeff Cheeger.
\newblock A lower bound for the smallest eigenvalue of the {L}aplacian.
\newblock In {\em Problems in analysis ({P}apers dedicated to {S}alomon
  {B}ochner, 1969)}, pages 195--199. Princeton Univ. Press, Princeton, N. J.,
  1970.

\bibitem{colbois2017compact}
Bruno Colbois, Ahmad El~Soufi, and Alexandre Girouard.
\newblock Compact manifolds with fixed boundary and large {S}teklov
  eigenvalues.
\newblock {\em Proc. Amer. Math. Soc.}, 147(9):3813--3827, 2019.

\bibitem{CGG}
Bruno Colbois, Alexandre Girouard, and Katie Gittins.
\newblock Steklov eigenvalues of submanifolds with prescribed boundary in
  {E}uclidean space.
\newblock {\em J. Geom. Anal.}, 29(2):1811--1834, 2019.

\bibitem{CGGS}
Bruno Colbois, Alexandre Girouard, Carolyn Gordon, and David Sher.
\newblock Some recent developments on the {S}teklov eigenvalue problem.
\newblock {\em Rev. Mat. Complut.}, 37(1):1--161, 2024.

\bibitem{CGH}
Bruno Colbois, Alexandre Girouard, and Asma Hassannezhad.
\newblock The {S}teklov and {L}aplacian spectra of {R}iemannian manifolds with
  boundary.
\newblock {\em J. Funct. Anal.}, 278(6):108409, 38, 2020.

\bibitem{colbois2016steklov}
Bruno Colbois, Alexandre Girouard, and Binoy Raveendran.
\newblock The {S}teklov spectrum and coarse discretizations of manifolds with
  boundary.
\newblock {\em Pure and Applied Mathematics Quarterly}, 14(2):357--392, 2018.

\bibitem{Daners}
Daniel Daners.
\newblock Domain perturbation for linear and semi-linear boundary value
  problems.
\newblock In {\em Handbook of differential equations: stationary partial
  differential equations. {V}ol. {VI}}, Handb. Differ. Equ., pages 1--81.
  Elsevier/North-Holland, Amsterdam, 2008.

\bibitem{Escobar1}
Jos\'{e}~F. Escobar.
\newblock The geometry of the first non-zero {S}tekloff eigenvalue.
\newblock {\em J. Funct. Anal.}, 150(2):544--556, 1997.

\bibitem{Escobar2}
Jos\'{e}~F. Escobar.
\newblock An isoperimetric inequality and the first {S}teklov eigenvalue.
\newblock {\em J. Funct. Anal.}, 165(1):101--116, 1999.

\bibitem{GrayTubes}
Alfred Gray.
\newblock {\em Tubes}, volume 221 of {\em Progress in Mathematics}.
\newblock Birkh\"{a}user Verlag, Basel, second edition, 2004.
\newblock With a preface by Vicente Miquel.

\bibitem{AsmaMiclo}
Asma Hassannezhad and Laurent Miclo.
\newblock Higher order {C}heeger inequalities for {S}teklov eigenvalues.
\newblock {\em Ann. Sci. \'{E}c. Norm. Sup\'{e}r. (4)}, 53(1):43--88, 2020.

\bibitem{JammesStek}
Pierre Jammes.
\newblock Une in\'egalit\'e de {C}heeger pour le spectre de {S}teklov.
\newblock {\em Annales de l'Institut Fourier}, 65(3):1381--1385, 2015.

\bibitem{LeeSM}
John~M. Lee.
\newblock {\em Introduction to smooth manifolds}, volume 218 of {\em Graduate
  Texts in Mathematics}.
\newblock Springer, New York, second edition, 2013.

\bibitem{Payne1970}
Lawrence~E. Payne.
\newblock Some isoperimetric inequalities for harmonic functions.
\newblock {\em SIAM J. Math. Anal.}, 1:354--359, 1970.

\bibitem{PStek2018}
H\'{e}l\`ene Perrin.
\newblock Lower bounds for the first eigenvalue of the {S}teklov problem on
  graphs.
\newblock {\em Calc. Var. Partial Differential Equations}, 58(2):Art. 67, 12,
  2019.

\bibitem{SWY}
Richard Schoen, Scott Wolpert, and Shing~Tung Yau.
\newblock Geometric bounds on the low eigenvalues of a compact surface.
\newblock In {\em Geometry of the {L}aplace operator ({P}roc. {S}ympos. {P}ure
  {M}ath., {U}niv. {H}awaii, {H}onolulu, {H}awaii, 1979)}, Proc. Sympos. Pure
  Math., XXXVI, pages 279--285. Amer. Math. Soc., Providence, R.I., 1980.

\bibitem{Xiong2022}
Changwei Xiong.
\newblock On the spectra of three {S}teklov eigenvalue problems on warped
  product manifolds.
\newblock {\em J. Geom. Anal.}, 32(5):Paper No. 153, 35, 2022.

\bibitem{YauIsopcon}
Shing~Tung Yau.
\newblock Isoperimetric constants and the first eigenvalue of a compact
  {R}iemannian manifold.
\newblock {\em Ann. Sci. \'{E}cole Norm. Sup. (4)}, 8(4):487--507, 1975.

\end{thebibliography}
\bibliographystyle{plain}

\bigskip
\footnotesize

\textsc{Université de Neuchâtel, Institut de mathématiques, Rue Emile-Argand 11, CH-2000 Neuchâtel, Switzerland}\par\nopagebreak
\textit{E-mail address}, \texttt{heleneperrin19@gmail.com}

\end{document}